\documentclass[12pt,reqno]{amsart}

\usepackage{epsfig}
\usepackage{amsmath, amsthm, amsfonts, amssymb, mathrsfs}
\usepackage{tikz-cd}
\usepackage{graphicx}

\numberwithin{equation}{section}
\DeclareMathOperator\Hom{Hom}

\newcommand{\D}{{\mathbb D}}

\newcommand{\R}{{\mathbb R}}
\newcommand{\C}{{\mathbb C}}
\newcommand{\N}{{\mathbb N}}

\newcommand{\Z}{{\mathbb Z}}

\newtheorem{theo}{{\sc \bf Theorem}}[section]
\newtheorem{cor}[theo]{{\sc \bf Corollary}}
\newtheorem{lem}[theo]{{\sc \bf Lemma}}
\newtheorem{prop}[theo]{{\sc \bf Proposition}}

\begin{document}

\title{Rapid Decay for Odometers}

\author[Klimek]{Slawomir Klimek}
\address{Department of Mathematical Sciences,
Indiana University Indianapolis,
402 N. Blackford St., Indianapolis, IN 46202, U.S.A.}
\email{klimek@iu.edu}

\author[McBride]{Matt McBride}
\address{Department of Mathematics and Statistics,
Mississippi State University,
175 President's Cir., Mississippi State, MS 39762, U.S.A.}
\email{mmcbride@math.msstate.edu}

\thanks{We would like to thank A. Rennie for helpful comments.}

\date{\today}

\begin{abstract}
We discuss rapid decay functions on odometer Cantor spaces and their noncommutative geometry applications.
\end{abstract}

\maketitle

\section{Introduction}

Odometers are well studied \cite{D} examples of minimal dynamical systems. They are also compact Abelian groups classified by a supernatural number $S$ and denoted by $\Z_S$ in this paper. Thus, following Pontryagin \cite{M}, they admit a general theory of distributions and Fourier transform. Because odometers are Cantor sets, the role of test functions is played by locally constant functions. The aim of this paper is to introduce and study more robust classes of test functions on odometers, namely classes of ``rapid decay" functions.

The Pontryagin dual $\widehat{\Z_s}$ of the odometer space is an infinite discrete, pure torsion Abelian group. Given a discrete group endowed with a length function the property of rapid decay is a general notion that was introduced and studied by Jolissaint in \cite{J} in the context of group C$^*$-algebras. However, $\widehat{\Z_s}$ is not finitely generated and so there is no word length function usually needed to define rapid decay property.

To overcome this issue we introduce and classify non-archimedean length functions with a growth condition on $\widehat{\Z_S}$. Then, given such a length function, we define and study rapid decay functions on odometers using Fourier transform in the usual way. 

Several aspects of noncommutative geometry and topology require understanding differentiable structures in C$^*$-algebras \cite{Connes}. Those are speciﬁed by a dense $*$-subalgebra satisfying smoothness properties such as completeness under a suitable locally convex topology, spectral stability, closure under holomorphic and smooth functional calculi. Very many examples of such smooth subalgebras of C$^*$-algebras have been discovered and investigated. There are at least two fairly general approaches to constructions of such smooth algebras; one originated by Blackadar and Cuntz \cite{BC} and another general approach is due to Kissin and Shulman \cite{KS}, see also \cite{BKS}. 
We verify that the algebras of rapid decay functions form smooth subalgebras of the algebra of continuous functions on $\Z_S$ even though they do not fit the above mentioned general schemes.

As a dynamical system application of rapid decay functions we consider the cohomological equation associated to the odometer's minimal dynamical system. This equation does not have an interesting answer in the category of continuous functions however it is very manageable for rapid decay functions.

Smooth subalgebras also play a role in K-theoretic considerations. Because of stability under holomorphic functional calculus, a smooth subalgebra has the same K-theory as the corresponding C$^*$-algebra and so the generators can be constructed from smooth elements of the algebra. Also, the commutator condition in the definition of a spectral triple requires a dense $*$-subalgebra. Spectral triples can be used to represent K-homology classes of a C$^*$-algebra. In the paper we discuss in some detail K-theory and K-homology of odometers with emphasis on rapid decay subalgebras.

Odometers are the easiest class of C$^*$-algebras classifiable by a supernatural number. Others include Bunce-Deddence algebras \cite{BD1},  \cite{BD2}, Uniformly Hyperfinite Algebras \cite{Dav} and possibly more. We anticipate that the methods developed in this paper will be very useful in similar studies of those C$^*$-algebras.

The paper is organized as follows. Section 2 contains an overview of basic properties of infinite compact monothetic groups i.e the odometers, mostly following reference . In section 3 we introduce and classify non-archimedean length functions on duals of odometers.  Section  4 describes generalities of smooth subalgebras of C$^*$-algebras. In section 5 we introduce a class of rapid decay functions on odometers and prove that they form a smooth subalgebra of the C$^*$-algebra of continuous functions. Section 6 is devoted to rapid decay version of dynamical cohomology of the odometer shift. We conclude with K-theoretic remarks in section 7.

\section{Monothetic Groups}
\subsection{Odometers}
A Hausdorff topological group $G$ is called {\it monothetic} if it has a dense cyclic subgroup. In this paper we only consider the case of compact $G$.  It follows immediately that $G$ is Abelian and separable.  We start by describing the structure of such groups following \cite{HS}.  The key tool is the character (dual) group and Pontryagin duality, which translates properties of groups into properties of their duals. A good, concise book on Pontryagin duality is \cite{M}.

Let $S^1$ be the unit circle:
\begin{equation*}
S^1 = \{z\in\C : |z|=1\},
\end{equation*} 
and let $\widehat{G}$ denote the dual group $G$, the group of continuous homomorphisms from $G$ to $S^1$ equipped with the compact-open topology.  It is well known \cite{M} that if $G$ is compact then $\widehat{G}$ is discrete.  

We typically use additive notation for an Abelian group, however we use multiplicative notation for the dual group.  Given a monothetic group $G$, let $x_1$ be a generator of a dense cyclic subgroup, and we set $x_n=nx_1$ for $n\in\Z$, so that $x_0:=0$ is the neutral element of $G$. Then we can identify the dual group $\widehat{G}$ of $G$ with a discrete subgroup $\widetilde{G}$ of $S^1$ via the map given by:
\begin{equation*}
\widehat{G}\ni\chi\mapsto \chi(x_1)\in \widetilde{G}.
\end{equation*}  
Conversely, using Pontryagin duality, if $H$ is a discrete subgroup of $S^1$, then $H$ is the dual group of a compact monothetic group, namely $\widehat{H}$, see \cite{HS}.

In this paper we only consider the case when $\widehat{G}$ is of pure torsion.
We have, see \cite{M}, that $G$ is totally disconnected if and only if $\widehat{G}$ is of pure torsion.

Next, we consider odometers and their relations to monothetic groups. Further details on odometers can be found in \cite{D}.
The definition of an odometer utilized in this paper uses scales. Let $s = (s_m)_{m\in\N}$ be a sequence of positive integers such that $s_m$ divides $s_{m+1}$ and $s_m<s_{m+1}$.  There are natural homomorphisms between the consecutive finite cyclic groups $\Z/{s_{m+1}}\Z \to \Z/s_m\Z$, namely congruence modulo $s_m$. Thus the inverse limit:
$$G_s = \lim_{\underset{m\in \mathcal \N}{\longleftarrow}} \mathbb Z/ s_m\mathbb Z$$
is well defined as the subset of the countable product $\prod_m \Z/s_m\Z$ consisting of sequences $(y_1, y_2, y_3, \ldots)$ such that $y_{m+1}\equiv y_m\ (\textrm{mod } s_m)$. $G_s$ becomes a topological group when endowed with the product topology over the discrete topologies in $\Z/s_m\Z$. With our assumptions, this group is a Cantor set \cite{W}. 
In fact, $G_s$ is a compact, totally disconnected monothetic group with the generator $x_1$ of a cyclic subgroup is given by $(1,1,1,1,\ldots).$

A {\it supernatural number} $S$ is defined as the formal product: 
\[S= \prod_{p-\textnormal{prime}} p^{\epsilon_p}, \;\;\; \epsilon_p \in\{0,1, \cdots, \infty\}.\]
If $\sum \epsilon_p < \infty$ then $S$ is said to be a finite supernatural number (a regular natural number), otherwise it is said to be infinite.  Supernatural numbers can be multiplied and there is a usual divisibility theory of such numbers that includes concepts of the greatest common divisor and the least common multiple.

Given a scale $s= (s_m)_{m\in\N}$ we define the corresponding supernatural number $S$ to be the least common multiple of all $s_m$'s:
\begin{equation*}
S=\mathrm{lcm}(s_m)
\end{equation*}
It follows that $s_m$'s are divisors of $S$. It turns out that odometers are classified by the supernatural number $S$, see \cite{D}. Thus, groups $G_s$ and $G_{s'}$ are isomorphic if and only is the corresponding supernatural numbers $S$ and $S'$ are equal. In view of this, if $s$ is a scale and $S$ is the corresponding supernatural number we write:
\begin{equation*}
\Z_S:=G_s.
\end{equation*}
Also, we simply write 1 for the generator of the cyclic group $\Z$ in $\Z_S$.

\subsection{Odometer Dual} We discuss in more detail the structure of the dual group $\widetilde{\Z_S}\subseteq S^1$. It follows from the definition of the odometer that:
\begin{equation*}
\widetilde{\Z_S}=\{z\in\C: \textrm{ there is }s|S \textrm{ such that }z^s=1\}.
\end{equation*}
On the other hand, any subgroup of $S^1$ consisting of roots of unity is of the above form for a supernatural number, namely the least common multiple of the orders of all elements.
Consequently, any infinite compact, totally disconnected, monothetic group is an odometer $\Z_S$. 

Any element $z\ne 1$ of $\widetilde{\Z_S}$ can be uniquely as
\begin{equation*}
z=e^{\frac{2\pi ik}{s}}
\end{equation*}
where $s$, the order of $z$, divides $S$ and $0<k<s$ is relatively prime with $s$. Following Knuth we write $k\perp s$ whenever $k$ is relatively prime with $s$.

Alternatively, if $s = (s_m)_{m\in\N}$ is a scale, any element $z\ne 1$ of $\widetilde{\Z_S}$ can be uniquely as
\begin{equation}\label{z_scale_rep}
z=e^{\frac{2\pi ij}{s_m}}
\end{equation}
where $0<j<s_m$ and $\frac{s_m}{s_{m-1}}$ does not divide $j$.

The main structural result about $\widetilde{\Z_S}$ is contained in the following elementary statement.
\begin{prop} \label{SubgroupsProp}
Every finite subgroup of $\widetilde{\Z_S}$ is cyclic. For every divisor $s$ of $S$ there is a unique subgroup of order $s$ and it is generated by $e^{\frac{2\pi i}{s}}$.
\end{prop}
\begin{proof}
Let $G$ be a finite subgroup of $\widetilde{\Z_S}$ and let $s$ be the least common multiple of orders of all elements of $G$. Then $G$ is a subgroup of the cyclic group $\{e^{2\pi ij/s}:j=0,\ldots,s-1\}$ since all roots of unity of order dividing $s$ must be of the form $e^{2\pi ij/s}$. It follows that $G$, as a subgroup of a finite cyclic group, is also cyclic. By Lagrange's theorem the least common multiple of orders of all elements of $G$ is equal to its order and thus the order of $G$ is $s$. Consequently, $G$ is the cyclic group $\{e^{2\pi ij/s}:j=0,\ldots,s-1\}$ which is generated by $e^{2\pi i/s}$.
\end{proof}

\section{Length Functions}

The discrete group $\widetilde{\Z_S}$ is not finitely generated so word length functions are not available. Instead, we consider general non-archimedean length functions on $\widetilde{\Z_S}$ with an additional growth condition. They are defined as functions: $\lambda: \widetilde{\Z_S}\to [1,\infty)$ satisfying the following conditions:
\begin{enumerate}
\item $\lambda(z)=1$ iff and only if $z=1$ (normalization)
\item $\lambda(z_1z_2)\leq\text{max}\{\lambda(z_1), \lambda(z_2)\}$ (non-archimedean property)
\item For every $r\geq 1$ the set $\{z\in\widetilde{\Z_S}: \lambda(z)\leq r\}$ is finite (growth condition)
\end{enumerate}
We choose to normalize our length functions to be 1 for the unity of the group which can be done without the loss of generality since adding a constant to a length function does not change the non-archimedean property.

Elementary properties of  non-archimedean length functions on $\widetilde{\Z_S}$ are summarized in the following statement.
\begin{prop} Suppose $\lambda$ is a non-archimedean length functions on  $\widetilde{\Z_S}$. It has the following properties for all  and $z,z_1,z_2\in \widetilde{\Z_S}$:
\begin{enumerate}
\item $\lambda(z^k)\leq \lambda(z)$ for every $k\in\Z$
\item  if ord$(z_1)=$ ord$(z_2)$ then  $\lambda(z_1)=\lambda(z_2)$
\item  $\lambda(z^{-1})=\lambda(z)$
\item   if  $\lambda(z_2)<\lambda(z_1)$ then  $\lambda(z_1z_2)=\lambda(z_1)$
\end{enumerate}
\end{prop}
\begin{proof} We prove the properties in order stated in the proposition.

Property 1: Is a consequence of repeated use of non-archimedean property.

Property 2:  If ord$(z_1)=$ ord$(z_2)=s$ then
\begin{equation*}
z_1=e^{\frac{2\pi ik_1}{s}}\textrm{ and }z_2=e^{\frac{2\pi ik_2}{s}}
\end{equation*}
where $k_1,k_2\perp s$. If $k_1',k_2'$ are inverses of $k_1,k_2$ modulo $s$ then using Property 1 we have
\begin{equation*}
\lambda(z_2) = \lambda(z_1^{k_1'k_2})\le \lambda(z_2).
\end{equation*}
Similarly, $\lambda(z_1)\le\lambda(z_2)$.

Property 3: This is a consequence of the previous property.

Property 4: Notice that
\begin{equation*}
\lambda(z_1z_2)\le\textrm{max }\{\lambda(z_1),\lambda(z_2)\} = \lambda(z_1)
\end{equation*}
by assumption.  On the other hand, using Property 3, we have that
\begin{equation*}
\lambda(z_1) = \lambda(z_2z_2^{-1}z_1)\le\textrm{max }\{\lambda(z_2^{-1}),\lambda(z_1z_2)\}=\textrm{max }\{\lambda(z_2),\lambda(z_1z_2)\} = \lambda(z_1z_2)
\end{equation*}
where the last equality is due to the assumption.  It now follows that $\lambda(z_1)=\lambda(z_1z_2)$.
\end{proof} 

The basic examples of length functions are constructed as follows. First choose a scale  $s = (s_m)_{m\in\N}$ for the supernatural number $S$, which  is a sequence of positive integers such that $s_m$ divides $s_{m+1}$, $s_m<s_{m+1}$, and such that $S=\mathrm{lcm}(s_m)$. Set $s_0:=1$. Then choose an increasing sequence $l = (l_m)_{m\in\N}$ of numbers in $(1,\infty)$ such that $\lim_{n\to\infty}l_n=\infty$. Set $l_0:=1$. Given $s$ and $l$ the corresponding length function $\lambda_{s,l}$ is defined as follows. If $z=1$ we set $\lambda_{s,l}(1):=l_0=1$. If $z\ne 1$, using representation \eqref{z_scale_rep} of elements of $\widetilde{\Z_S}$, we set:
\begin{equation*}
\lambda_{s,l}(z)=\lambda_{s,l}\left(e^{\frac{2\pi ij}{s_m}}\right):=l_m,
\end{equation*}
for all $m\in \N$ and $0<j<s_m$ such that $\frac{s_m}{s_{m-1}}$ does not divide $j$. 
The assumption that $\lim_{n\to\infty}l_n=\infty$ implies the growth property of $\lambda_{s,l}$.

To verify the non-archimedean property suppose $z_1=e^{\frac{2\pi ij}{s_m}}$ and $z_2=e^{\frac{2\pi ik}{s_n}}$ with $s_m$ dividing $s_n$. Then we have:
\begin{equation*}
z_1z_2=e^{\frac{2\pi i(k+js_n/s_m)}{s_n}},
\end{equation*}
so, because $k+js_n/s_m$ is an integer, $\lambda_{s,l}(z_1z_2)$ is at most $s_n$.

It turns out that $\lambda_{s,l}$ encompass all possible non-archimedean length functions on $\widetilde{\Z_S}$.

\begin{theo} (Classification of length functions)  Suppose $\lambda$ is a non-archimedean length functions on  $\widetilde{\Z_S}$. Then there exist a scale $s = (s_m)_{m\in\N}$ and an increasing sequence $l = (l_m)_{m\in\N}$ of numbers in $(1,\infty)$ with $\lim_{n\to\infty}l_n=\infty$ such that $\lambda=\lambda_{s,l}$.
\end{theo}
\begin{proof}
Notice that the set $\{z\in\widetilde{\Z_S}: l(z)\leq r\}$ is a finite subgroup of $\widetilde{\Z_S}$. This follows from the non-archimedean property in the definition of a length function and from Property 3 of the proposition above. 

Let $l = (l_m)_{m\in\\Z_{\geq 0}}$ with $l_0=1$ be the increasing sequence of values of $\lambda$. Then
\begin{equation}\label{Gmdef}
G_m:= \{z\in\widetilde{\Z_S}: l(z)\leq l_m\}
\end{equation}
are finite subgroups of $\widetilde{\Z_S}$ and $G_m\subsetneq G_{m+1}$.
From Proposition \ref{SubgroupsProp} the orders $s_m=|G_m|$ of $G_m$ divide $S$. Since $G_m$ is a nontrivial subgroup of $G_{m+1}$, $s_m$ divides $s_{m+1}$ and $s_m<s_{m+1}$. Also, any element of $\widetilde{\Z_S}$ is in one of the subgroups so its order divides $s_m$ for some $m$. It follows that $\mathrm{lcm}(s_m)=S$ and so such a constructed sequence $s = (s_m)_{m\in\N}$ is a scale for the supernatural number $S$.

From the definition of the groups $G_m$, if $z\in G_m$ but $z\notin  G_{m-1}$, then $\lambda(z)=\lambda_m$. By Proposition \ref{SubgroupsProp} the structure of $G_m$ is:
\begin{equation*}
G_m=\{e^{2\pi ij/s_m}:j=0,\ldots,s_m-1\}
\end{equation*}
and so if $z\in G_m$ but $z\notin  G_{m-1}$, then $z$ is of the form $z=e^{2\pi ij/s_m}$ with $0<j<s_m$ such that $\frac{s_m}{s_{m-1}}$ does not divide $j$.
It follows that $\lambda=\lambda_{s,l}$.
\end{proof}

\section{Generalities on Smooth Subalgebras}

In this section we discuss a definition and several general results about smooth subalgebras of C*-algebras. Let $A$ be a unital C$^*$-algebra with norm $\|\cdot\|$ and let  $\mathcal{A}\subseteq A$
be a  dense *-subalgebra of $A$ which is a Frechet *-algebra with respect to a locally convex topology stronger than that of $A$. We will always suppose that we can define the Frechet topology of $\mathcal{A}$ using a countable collection of submultiplicative seminorms $\|\cdot\|_N$, $N=0,1,\ldots$. We call $\mathcal{A}$ a {\it Frechet subalgebra} of $A$.

There are several natural concepts of stability of a subalgebra. We say that $\mathcal{A}$ is {\it spectrally stable} if for any element $a$ of $\mathcal{A}$ its spectrum in $\mathcal{A}$ is the same as its spectrum in $A$. Equivalently, from the definition of the spectrum, $\mathcal{A}$ is spectrally stable if for any element $a\in \mathcal{A}$ that is invertible in $A$ its inverse $a^{-1}$ is in $\mathcal{A}$.

We say that $\mathcal{A}$ is {\it stable under the holomorphic calculus} if for any $a\in \mathcal{A}$ and a function $f$ that is holomorphic on an open domain containing the spectrum of $a$ we have $f(a)\in \mathcal{A}$. By considering the holomorphic (away from zero) function $\zeta\mapsto 1/\zeta$, we see that the stability under the holomorphic calculus is a stronger condition than the spectral stability. Notice however that if $\mathcal{A}$ a Frechet subalgebra of $A$ then spectral stability implies stability under the holomorphic calculus.
The usual arguments \cite{Bo} work here since we can write $f(a)$ as
\begin{equation*}
f(a) = \frac{1}{2\pi i}\int_C f(\zeta)(\zeta- a)^{-1}\,d\zeta\,.
\end{equation*}
where $C$ a contour around the spectrum of $a$. The convergence of the integral is guaranteed by the completeness of $\mathcal{A}$ and so $f(a)\in \mathcal{A}$.

We say that $\mathcal{A}$ is {\it stable under the smooth functional calculus} of self-adjoint elements if for any self-adjoint element $a$ of $\mathcal{A}$ and a smooth function $f(x)$ defined on an open neighborhood of the spectrum $\sigma(a)$ of $a$ we have $f(a)$ is in $\mathcal{A}$. In our situation of interest, when  $\mathcal{A}$ a Frechet subalgebra of $A$, we have the following elementary observation.

\begin{prop} Suppose $\mathcal{A}$ a Frechet subalgebra of a C$^*$-algebra $A$. If $\mathcal{A}$ is stable under the smooth functional calculus of self-adjoint elements then $\mathcal{A}$ is stable under the holomorphic calculus.
\end{prop}
\begin{proof} We only need to verify that $\mathcal{A}$ is spectrally stable. Assume first that $a\in\mathcal{A}$ is positive and invertible in $A$.  Then for some $\varepsilon>0$ we have that $\sigma(a)\subseteq[\varepsilon,\infty)$.  Notice that $x\mapsto 1/x$ is a smooth function on $[\varepsilon,\infty)$, thus $a^{-1}\in\mathcal{A}$.

Next assume $a$ is just invertible in $A$.  Generally theory implies that $a^*$ is invertible and hence $a^*a$ is invertible and positive.  So from the first part, $(a^*a)^{-1}\in\mathcal{A}$.  But $a^{-1}=(a^*a)^{-1}a^*$, and so $a^{-1}\in\mathcal{A}$ since $\mathcal{A}$ is a $*$-subalgebra.
\end{proof}

The above proposition motivates now the following definition: a  dense $*$-subalgebra $\mathcal{A}\subseteq A$ is called {\it smooth} if it is a Frechet subalgebra of $A$ which is stable under the smooth functional calculus of self-adjoint elements.

In general it is not easy to establish stability under smooth functional calculus. More information about Frechet seminorms is required, submultiplicativity is not enough. In a typical approach, like in the above Cauchy formula expressing holomorphic functions through resolvents, we can build up any smooth function from exponentials.
This motivates the following definition. 

We say that $\mathcal{A}$ has {\it polynomially bounded exponentials} if for every self-adjoint $a\in\mathcal{A}$ and every $N=0,1,\ldots$ there is an integer $k$ and a positive constant $C$ (both depending on $a$ and $N$) such that for every $t\in\R$ we have:
\begin{equation*}
\left\| e^{ita}\right\|_N\leq C(1+|t|)^k.
\end{equation*}

\begin{prop} Suppose that a Frechet subalgebra $\mathcal{A}\subseteq A$ has polynomially bounded exponentials. Then $\mathcal{A}$ is closed under smooth functional calculus of self-adjoint elements, in other words, it is a smooth subalgebra of $A$.
\end{prop}
\begin{proof} Let $a$ be a self-adjoint element of $\mathcal{A}$.
Suppose $f(x)$ is a smooth function defined on an open neighborhood of the spectrum of $a$.
Without a loss of generality we can assume that $f(x)$ is smooth on $\R$ and is $L$-periodic: $f(x+L)=f(x)$ for some $L$. Then $f(x)$ admits a Fourier series representation with rapid decay coefficients $\{f_n\}$, and hence we have
\begin{equation*}
f(a)=\sum_{n\in\Z}f_ne^{2\pi ina/L}.
\end{equation*}
Using the fact that $\mathcal{A}$ has polynomially bounded exponentials and estimating directly we get
\begin{equation*}
\|f(a)\|_N\le \sum_{n\in\Z}|f_n|\|e^{2\pi ina/L}\|_N\le C\sum_{n\in\Z}|f_n|(1+|n|)^k
\end{equation*}
for some $k$ and $C>0$.  The above sum is finite since $\{f_n\}$ is a rapid decay sequence.
\end{proof}

\section{RD Smooth Subalgebras}

From now on we impose further growth condition on length functions used below. We say that a non-archimedean length function $\lambda$ {\it grows fast enough} if there are positive numbers $\alpha$ and $c$ such that for every $z\in\widetilde{\Z_S}$ we have:
\begin{equation}\label{fast_enough}
\lambda(z)\geq c\,(\textrm{ord}(z))^\alpha,
\end{equation}
where ord$(z)$ is the order of the group element $z\in\widetilde{\Z_S}$.

Let $d(r)$ be the cardinality of the set $\{z\in\widetilde{\Z_S}: \lambda(z)\leq r\}$:
\begin{equation*}
d(r):=|\{z\in\widetilde{\Z_S}: \lambda(z)\leq r\}|.
\end{equation*}
In the analysis of rapid decay functions below we need growth estimates on $d(r)$.
\begin{lem}\label{est_dr} Suppose that a non-archimedean length function $\lambda$ grows fast enough. Then we have
\begin{equation*}
d(r)\leq Cr^\beta
\end{equation*}
for some positive constants $\beta$ and $C$.
\end{lem}\label{fastenoughlemma}
\begin{proof} Our proof uses the following information about the large $n$ asymptotics of totient summatory function (\cite{HW} edition 6, Theorem 330):
\begin{equation*}
\Phi(n):=\sum_{k\leq n}\phi(k)\sim\frac{3}{\pi^2}n^2.
\end{equation*}
For a fixed $r$, observe the following:
\begin{equation*}
\begin{aligned}
|\{z\in\widetilde{\Z_S}:\textrm{ord}(z)\le r\}| &= \sum_{k\le r, k|S}|\{z\in S^1:\textrm{ord}(z)=k\}|\le\sum_{k\le r}|\{z\in S^1:\textrm{ord}(z)=k\}|\\
&=\sum_{k=1}^r\phi(k)\sim\frac{3}{\pi^2}r^2.
\end{aligned}
\end{equation*}
On the other hand, using the inequality in equation \eqref{fast_enough} and the above observation, we have that
\begin{equation*}
d(r)\le |\{z\in\widetilde{\Z_S}:c(\textrm{ord}(z))^\alpha\le r\}|\le Cr^\beta
\end{equation*}
for some positive constant $C$ and where $\beta=2/\alpha$.
\end{proof}

If $z\in\widetilde{\Z_S}$, let $\chi_z\in\widehat{\Z_S}$ be the corresponding character functions on $\Z_S$ so that
\begin{equation*}
\chi_z(1)=z.
\end{equation*}
Let $dx$ be the normalized Haar measure on $\Z_S$. To any continuous function $f$ on $\Z_S$ one can associate the usual Fourier series:
\begin{equation}\label{FSeries}
\sum_{z\in\widetilde{\Z_S}}\hat f_z \chi_z,
\end{equation}
where
\begin{equation*}
\hat f_z=\int_{\Z_S}f(x) \chi_{\bar z}(x)\,dx.
\end{equation*}
The series is not always convergent but the Fourier coefficients $\hat f_z$ determine $f$.

We define norms on the space of series given by equation \eqref{FSeries}  as follows:
\begin{equation*}
\|f\|_N:=\sum_{z\in\widetilde{\Z_S}}|\hat f_z|\lambda(z)^N.
\end{equation*}
Observe that we have the following property for every $N$:
\begin{equation*}
\|f\|_N\leq \|f\|_{N+1}.
\end{equation*}
Notice also that for complex conjugation we have:
\begin{equation*}
\|\bar f\|_N=\|f\|_N.
\end{equation*}
Additionally, if $\|f\|_N\leq\infty$, the series in equation \eqref{FSeries} converges uniformly to a continuous function on $\Z_S$ and we have:
 \begin{equation*}
\|f\|\leq \|f\|_0,
\end{equation*}
where $\|f\|$ is the usual sup norm on $C(\Z_S)$.

With those preliminaries we define  $C_{RD}(\Z_S)$ as
\begin{equation*}
 C_{RD}(\Z_S):=\{f\in C(\Z_S): \|f\|_N<\infty \textrm{ for every } N\}.
\end{equation*}

\begin{theo} $C_{RD}(\Z_S)$ is a smooth subalgebra of the $C^*$-algebra $C(\Z_S)$.
\end{theo}
\begin{proof} From the definition of $C_{RD}(\Z_S)$ it follows that it is a Frechet space closed under complex conjugation of functions. Since trigonometric polynomials are dense in $C(\Z_S)$, see \cite{M}, $C_{RD}(\Z_S)$ is also dense in $C(\Z_S)$. We need to verify submultiplicativity.  Let $f,g\in C_{RD}(\Z_S)$ and write them in there Fourier decomposition, that is
\begin{equation*}
f = \sum_{z\in\widetilde{\Z_S}}\hat{f}_z\chi_z\quad\textrm{and}\quad g = \sum_{w\in\widetilde{\Z_S}}\hat{g}_w\chi_w
\end{equation*}
and consider the following
\begin{equation*}
\begin{aligned}
\|fg\|_N &=\left\|\sum_{z,w}\hat{f}_z\hat{g}_w\chi_z\chi_w\right\|_N =\left\|\sum_{z,w}\hat{f}_z\hat{g}_w\chi_{zw}\right\|_N = \left\|\sum_u\left(\sum_{zw=u}\hat{f}_z\hat{g}_w\right)\chi_u\right\|_N\\
&=\sum_u\left|\sum_{zw=u}\hat{f}_z\hat{g}_w\right|\lambda(u)^N\le\sum_{z,w}|\hat{f}_z||\hat{g}_w|\lambda(zw)^N\\
&\le\sum_{z,w}|\hat{f}_z||\hat{g}_w|\lambda(z)^N\lambda(w)^N=\|f\|_N\|g\|_N.
\end{aligned}
\end{equation*}

It remains to show that $C_{RD}(\Z_S)$ has polynomially bounded exponentials.  Using the above Fourier decomposition for $f$, define the following
\begin{equation*}
f_{\le m} = \sum_{\{z:\lambda(z)\le l_m\}}\hat{f}_z\chi_z\quad\textrm{and}\quad f_{> m} = \sum_{\{z:\lambda(z)> l_m\}}\hat{f}_z\chi_z
\end{equation*}
for some natural number $m$.  Notice that for each natural number $m$ we have that 
$$f=f_{\le m} + f_{> m}.$$

Let $n$ be a nonzero integer and $f\in C_{RD}(\Z_S)$ real, $f=\overline{f}$, we need to show that $\|e^{inf}\|_N$ is bounded by a polynomial in $n$. Since $\lambda(z^{-1})=\lambda(z)$, it follows that 
$$\overline{f}_{\le m} = f_{\le m}\textrm{ and }\overline{f}_{>m}=f_{> m}.$$

Notice that by the properties of the $N$-norm we have that
\begin{equation*}
l_{m+1}\cdot\|f_{> m}\|_N = l_{m+1}\sum_{\lambda(z)>l_m}|\hat{f}_z|\lambda(z)^N\le\sum_{\lambda(z)>l_m}|\hat{f}_z|\lambda(z)^{N+1}\le\|f\|_{N+1}.
\end{equation*}

Using the fact that $C(\Z_S)$ is commutative, the submultiplicativity of norms in $C_{RD}(\Z_S)$ and the above observation, we have that
\begin{equation*}
\begin{aligned}
\|e^{inf}\|_N &= \|e^{inf_{\le m}}e^{inf_{> m}}\|_N\le\|e^{inf_{\le m}}\|_N\|e^{inf_{> m}}\|_N\\
&\le\|e^{inf_{\le m}}\|_Ne^{|n|\|f_{> m}\|_N}\le e^{\frac{|n|}{l_m}\|f\|_{N+1}}\|e^{inf_{\le m}}\|_N.
\end{aligned}
\end{equation*}

Next we want to estimate $\|e^{inf_{\le m}}\|_N$.  Notice that the set $\{z:\lambda(z)\le l_m\}$ in the definition of $f_{\le m}$ is the group $G_m$ of equation \eqref{Gmdef}. The functions on $\Z_S$ which are given by a finite Fourier series sum over characters in $G_m$ can be alternatively characterized as those functions $f$ such that for every $x\in\Z_S$
\begin{equation*}
f(x+s_m)=f(x).
\end{equation*}
In particular, the exponential 
$g_m:=e^{inf_{\le m}}$ is also such a function and so has a  finite Fourier series expansion:
\begin{equation*}
g_m=e^{inf_{\le m}}=\sum_{z\in G_m}(\widehat{g_m})_z\chi_z.
\end{equation*}
Moreover, the Fourier coefficients $(\widehat{g_m})_z$ satisfy the following inequality
\begin{equation*}
(\widehat{g_m})_z= \left|\int_{\Z_S}e^{inf_{\le m}}\overline{\chi}_z(x)\ dx\right|\le 1.
\end{equation*}
Consequently, we have that
\begin{equation*}
\|e^{inf_{\le m}}\|_N\le \sum_{z\in G_m} 1\cdot\lambda(z)^N\le s_m\,l_m^N.
\end{equation*}
This might be considered as a conceptual crux of the proof: the exponentials can be conveniently estimated on a dense subalgebra of $C(\Z_S)$ of functions with finite Fourier series expansion.

Returning to the estimate on the norm $\|e^{inf}\|_N $, we now have:
\begin{equation*}
\|e^{inf}\|_N \le e^{\frac{|n|}{l_m}\|f\|_{N+1}}\|e^{inf_{\le m}}\|_N\le e^{\frac{|n|}{l_m}\|f\|_{N+1}}s_m\,l_m^N.
\end{equation*}
Using Lemma \ref{est_dr} for $r=l_m$ we get:
\begin{equation*}
\|e^{inf}\|_N \le e^{\frac{|n|}{l_m}\|f\|_{N+1}}s_m\,l_m^N \le C\,e^{\frac{|n|}{l_m}\|f\|_{N+1}}\,l_m^{N+\beta}.
\end{equation*}

The estimate above is valid for any $m$. Notice that for any $n$ with $|n|\ge 1$ there is a $m\in\Z_{\geq 0}$ such that
\begin{equation*}
l_m\le |n|\le l_{m+1}.
\end{equation*}
With this choice of $m$ we get the desired estimate on exponentials:
\begin{equation*}
\|e^{inf}\|_N \le C\,e^{\|f\|_{N+1}}\,|n|^{N+\beta},
\end{equation*}
finishing the proof.
\end{proof}

\section{Dynamical Cohomology}
We start this section by reviewing dynamical cohomology, following \cite{KH} and \cite{T}.
By a {\it topological dynamical system} $(X,\varphi)$, we mean a topological space $X$ and a continuous map $\varphi:X\to X$, see \cite{KH}.  A topological dynamical system $(X,\varphi)$ is called {\it topologically transitive} if there exists a point $x\in X$ such that its orbit $\{\varphi^n(x)\}_{n\in\Z}$ is dense in $X$.  $(X,\varphi)$ is called {\it minimal} if every orbit is dense in $X$.  Suppose that $\Z_S$ is an odometer with $1$ being the generator of a dense cyclic subgroup. Then we define the map $\varphi:\Z_S\to \Z_S$ by the formula:
$$\varphi(x) = x+1.$$  
It follows that $(\Z_S,\varphi)$ is a minimal system.  

A useful notion in the theory of dynamical systems is the notion of a cocycle. It is a function $\rho: \N \times X \to \C$ that obeys the cocycle equation
$$\rho(k+l, x) = \rho(k,x) + \rho(k,\varphi^l(x)) $$
for all $k,l \in \N$ and $x \in X$. Part of the definition is also to specify a class of functions involved, perhaps continuous, smooth or, as will be the case below, rapid decay. If $\varphi$ is invertible, then one can extend cocycles to be functions on $\Z \times X$.

Notice that $\rho(k, x)$ is completely determined by the function $r(x):= \rho(1,x)$, namely:
\begin{equation*}
\rho(k, x)=r(x)+r(\varphi(x))+\ldots+r(\varphi^{k-1}(x)).
\end{equation*}

The significance of cocycles is that they allow one to construct skew products $X \times_\rho \C$ of the original dynamical system $X$. Those are defined as the Cartesian product 
$$X \times_\rho \C:=\{ (x,c): x \in X, c \in \C \}$$ 
with the map 
$$\varphi_\rho(x,c) := (\varphi(x),c + \rho(1,x)).$$ 
This turns out to be a useful method to build new dynamical systems.

A special type of cocycle is a coboundary, which is a cocycle  $\rho: \N \times X \to \C$ of the form $\rho(k,x) := g(\varphi^k(x)) - g(x)$ for some function $g: X \to \C$.  
Alternatively, we have
\begin{equation*}
r(x):= g(\varphi(x)) - g(x)
\end{equation*}
for some $g: X \to \C$, which is called a cohomological equation and occurs in a variety of dynamical systems situations.
An extension $X \times_\rho \C$ of a dynamical system by a coboundary $\rho(k,x) = g(\varphi^k(x)) - g(x)$ can be conjugated to the trivial extension $X \times_0 \C$ by the change of variables $(x,c) \mapsto (x,c-g(x))$.

While every coboundary is a cocycle, the converse is not always true.   One can measure the extent to which this converse fails by introducing the first cohomology group 
$$H^1(X,\varphi) := Z^1(X,\varphi) / B^1(X,\varphi),$$ 
where $Z^1(X,\varphi)$ is the space of cocycles and $B^1(X,\varphi)$ is the space of coboundaries (note that both spaces are abelian groups).

Our next goal is to compute such first cohomology group for odometers. From now on $X=\Z_S$, $\varphi(x)=x+1$ and all the functions are assumed to be rapid decay, since considering continuous functions does not lead to a sensible answer.
Notice also that in Proposition 2.6 of \cite{KM6} the first cohomology was essentially computed when all functions are assumed to be trigonometric polynomials.

\begin{prop}\label{property_F}
Let $f\in C_{RD}(\Z_S)$.  Then $\int_{\Z_S}f\,dx=0$ if and only if there is $g\in C_{RD}(\Z_S)$ such that 
$$f = g\circ\varphi-g.$$
\end{prop}

\begin{proof}
By the invariance of Haar measure we have:
\begin{equation*}
\int_{\Z_S}(g\circ\varphi-g)\,dx=\int_{\Z_S}(g(x+1)-g(x))\,dx=0,
\end{equation*}
which verifies the necessary of the biconditional statement.

If $f\in C_{RD}(\Z_S)$ then $f$ has the following decomposition:
\begin{equation*}
f=\sum_{z\in\widetilde{\Z_S}}\hat f_z \chi_z
\end{equation*}
Notice that we have
\begin{equation*}
\int_{\Z_S}\chi_z(x)\,dx = \left\{
\begin{aligned}
&1 &&\textrm{if }z=1  \\
&0 &&\textrm{else,}
\end{aligned}\right.
\end{equation*}
which means that $\int_{\Z_S}\chi_z(x)\,dx =0$ if and only if $\hat f_1=0$.  

Assuming $\hat f_1=0$, the goal is to find a function $g\in C_{RD}(\Z_S)$ such that 
\begin{equation*}
g=\sum_{z\ne 1}\hat g_z \chi_z
\end{equation*}
and
\begin{equation}\label{cocyc}
f(x) = g(x+1)-g(x).
\end{equation}
Notice that for a nontrivial character we must have $\chi_z(1)=z\neq 1$.
Therefore, we can choose 
\begin{equation*}
\hat g_z = \frac{\hat f_z}{z-1},
\end{equation*}
which clearly satisfies equation \eqref{cocyc}.
We need to verify that $g\in C_{RD}(\Z_S)$.  By definition, we have that
\begin{equation*}
\|g\|_N = \sum_{z\neq 1}\frac{|\hat{f}_z|}{|z-1|}\lambda(z)^N
\end{equation*}
Notice that if $z=e^{2\pi ij/k}$ for $j\perp k$ and $k=\textrm{ord}(z)$, then, by basic trigonometry, there exists a constant $C>0$ such that 
$$\frac{1}{|z-1|}\le C\,\textrm{ord}(z).$$  
Thus, using the growth condition in equation \eqref{fast_enough}, we have an estimate
\begin{equation*}
\|g\|_N\le C\sum_{z\neq1}|\hat{f}_z|(\textrm{ord}(z))\lambda(z)^N\le C\sum_{z\neq 1}|\hat{f}_z|\lambda(z)^{1/\alpha}\lambda(z)^N \le C\|f\|_{N+1/\alpha}.
\end{equation*}
\end{proof}

\begin{cor}
With the above notation, the first cohomology group of rapid decay functions
$H_{RD}^1(\Z_S,\varphi)$ is isomorphic to $\C$.
\end{cor}
\begin{proof} By Proposition \ref{property_F}, the map $f\mapsto \int_{\Z_S}f\,dx$ is the required isomorphism.
\end{proof}

\section{K-Theoretic Considerations}
In this section we discuss the $K$-theory of the smooth subalgebras $C_{RD}(\Z_S)$. We also discuss the $K$-homology groups of $C(\Z_S)$.  We then find explicit Fredholm modules and spectral triples whose classes generate the $K$-homology.  
\subsection{K-Theory}
Note that, by earlier discussions, $C_{RD}(\Z_S)$ are closed under the holomorphic calculus. Hence, the inclusion $C_{RD}(\Z_S)\subseteq C(\Z_S)$ induces an isomorphism in $K$-Theory. 

Let $C(\Z_S,\Z)$ be the Abelian group of continuous functions on $\Z_S$ with values in $\Z$ and let $\mathcal{E}(\Z_S, \Z)$ be the Abelian group of locally constant functions on $\Z_S$ with values in $\Z$.
Since $\Z_S$ is totally disconnected, it follows from Exercise 3.4 in \cite{RLL} that
$$K_0(C(\Z_S)) \cong C(\Z_S,\Z)=\mathcal{E}(\Z_S, \Z).$$ 
The last equality above follows because the range $\Z$ is discrete. Additionally, from \cite{Dav} Example III.2.5, $C(\Z_S)$ is AF.
Therefore, by Exercise 8.7 in \cite{RLL}, we have:
$$K_1(C(\Z_S)) = 0.$$

It is well known (see exercise 7.7.5 in \cite{HR}) that $C(\Z_S, \Z)$ is a free Abelian group. Below we describe a set of generators.

Suppose $s= (s_m)_{m\in\N}$ is a scale for $\Z_S$. Then, similarly to the usual expansion of a $p$-adic integer, we can represent an arbitrary element $x$ in $\Z_S$ by 
\begin{equation*}
x=    \sum_{j=1}^\infty x_j s_{j-1} 
\end{equation*}
with $s_0 = 1$ and ``digits" $0\leq x_j< s_j/s_{j-1}$. In particular, $s_n$ divides $x$ if and only if $x_j=0$ for all $j\leq n$.

For $n=0,1,2, \ldots$ and $0\leq x<s_n$ define
\begin{equation*}
1_{(n,x)}(z) = \left\{
\begin{aligned}
&1 &&\textrm{if }s_n|(z-x)  \\
&0 &&\textrm{else.}
\end{aligned}\right.
\end{equation*}
Those are characteristic functions of open and closed subsets of $\Z_S$. Any two of those sets are either disjoint or one is a subset of the other. By density of $\Z_{\geq 0}$ in $\Z_S$, their characteristic functions separate points in $\Z_S$. It follows that linear combinations of functions $1_{(n,x)}$ are dense in $C(\Z_S)$. Additionally, by compactness of $\Z_S$, any function in $C(\Z_S, \Z)$ is a finite integer combination of $1_{(n,x)}$'s. However, $1_{(n,x)}$'s are not independent since
\begin{equation}\label{1(n,x)reductive}
1_{(n,x)} = \sum_{l=0}^{s_{n+1}/s_n-1} 1_{(n+1,x+ls_n)}.
\end{equation}

To proceed further we consider 
\begin{equation*}
1_{(x)}:= 1_{(n(x),x)}, 
\end{equation*}
where $n(x)$ is defined so that $s_{n(x)-1}\leq x<s_{n(x)}$ with 
\begin{equation*}
1_{(0)}:= 1_{(0,0)}=1.
\end{equation*}

\begin{prop} The collection of functions $\{1_{(x)}\}$ for all $x=0,1,2, \ldots$ form free generators of the group $C(\Z_S, \Z)$. In other words,
any element $f\in C(\Z_S, \Z)$ can be uniquely written as a finite sum:
\begin{equation*}
f=\sum_{x=0}^\infty f_{(x)}\,1_{(x)}
\end{equation*}
where $f_{(x)}\in\Z$.
\end{prop}
\begin{proof} 
As noticed above, linear combinations of the $1_{(n,x)}$'s, generate $C(\Z_S)$.  We also noticed that continuity of a function in $C(\Z_S,\Z)$ is equivalent to the function being locally constant.  Consequently, any function in $C(\Z_S,\Z)$ is a linear (integer) combinations of $1_{(n,x)}$'s.  Thus, to establish the existence part of the proposition we need to express functions $1_{(n,x)}$ in terms of functions $1_{(x)}$'s.  

We proceed by induction.  The ``base'' case of $n=0$ is immediate by definition.
%
For the inductive step, suppose $1_{(n,x)}$ can be expressed in terms of the $1_{(x)}$'s and consider the function $1_{(n+1,x)}$.  If $s_n\le x<s_{n+1}$, then by definition we have that $1_{(n+1,x)}=1_{(x)}$.  On the other hand if $0\le x<s_n$, by equation (\ref{1(n,x)reductive}) we have that
\begin{equation*}
1_{(n+1,x)} = 1_{(n,x)} - \sum_{l=1}^{s_{n+1}/s_n -1}1_{(n+1,x+ls_n)} = 1_{(n,x)} -\sum_{l=1}^{s_{n+1}/s_n-1}1_{(x+ls_n)},
\end{equation*}
which, by assumption, shows the inductive step and hence existence.

To establish uniqueness suppose that
\begin{equation*}
\sum_{x=0}^\infty f_{(x)}1_{(x)}=0
\end{equation*}
noting that the sum on the left-hand side is always a finite summation. We evaluate the sum on the left-hand side of the above equation at $0,1,2,\ldots$.  First we have
\begin{equation*}
\sum_{x=0}^\infty f_{(x)}1_{(x)}(0) = \sum_{x=0}^\infty f_{(x)}1_{(n(x),x)}(0)= 0.
\end{equation*}
Since $1_{(n(x),x)}(0) = 1$ if $s_{n(x)}$ divides $x$ for $s_{n(x)-1}\le x<s_{n(x)}$ and is zero else, from the above sum, it follows that $f_{(0)} = 0$.  Next consider
\begin{equation*}
0=\sum_{x=1}^\infty f_{(x)}1_{(x)}(1) = \sum_{x=1}^\infty f_{(x)}1_{(n(x),x)}(1)
\end{equation*}
for $s_{n(x)-1}\le x < s_{n(x)}$.  Like before note that $1_{(n(x),x)}(1) =1$ if $s_{n(x)}$ divides $x-1$ and is zero else.  Since $x-1<s_{n(x)}$, it follows that from the above sum that $f_{(1)}=0$.  Following this procedure inductively, we conclude that $f_{(x)}=0$ for all $x$ and hence the uniqueness follows.
\end{proof}

\subsection{K-Homology}
Since $C(\Z_S, \Z)$ is free using the Universal Coefficient Theorem we obtain 
$$K^1(C(\Z_S)) = 0 \textrm{ and } K^0(C(\Z_S)) = \Hom(C(\Z_S, \Z),\Z).$$

There are two classes of particularly interesting elements of $\Hom(C(\Z_S, \Z),\Z)$.
Define $e_{(y)}\in \Hom(C(\Z_S, \Z),\Z)$ for $0 \leq y\in \Z$ to be the ``dual" homomorphisms to the generators $1_{(x)}$ of $C(\Z_S, \Z)$:
\begin{equation*}
e_{(y)}(1_{(x)}) = \left\{
\begin{aligned}
&1 &&\textrm{if }y=x  \\
&0 &&\textrm{else,}
\end{aligned}\right.
\end{equation*}
The significance of $e_{(y)}$ is that any homomorphism $\Phi\in \Hom(C(\Z_S, \Z),\Z)$ can be uniquely written as a possibly infinite sum:
\begin{equation*}
\Phi=\sum_{y=0}^\infty\phi_{(y)}\,e_{(y)}
\end{equation*}
where $\phi_{(y)}\in\Z$. The infinite sum above is not a problem because when evaluating $\Phi$ on $f\in C(\Z_S, \Z)$ the resulting sum is always finite.

Another important class of homomorphisms are the evaluation maps $\delta_z$, $z=0,1,2,\ldots$:
\begin{equation*}
\delta_z(f):=f(z)
\end{equation*}
We have
\begin{equation*}
\delta_z(1_{(n,x)}) = \left\{
\begin{aligned}
&1 &&\textrm{if }s_n|(z-x)  \\
&0 &&\textrm{else.}
\end{aligned}\right.
\end{equation*}

To analyze this formula we need the following notation. If the natural number $n(x)$ is such that $s_{n(x)-1}\leq x<s_{n(x)}$ we set for $x\geq 1$, $x\in\Z$:
\begin{equation*}
\gamma(x):= x \textrm{ mod } s_{n(x)-1},
\end{equation*}
so that $0\leq \gamma(x)<  s_{n(x)-1}$.

The two classes of homomorphisms are closely related by the following formulas.
\begin{prop} With the above notation we have
\begin{equation*}
\delta_z=e_{(z)}+e_{(\gamma(z))}+e_{(\gamma^2(z))}+\ldots+e_{(0)} \textrm{ and } e_{(z)}=\delta_z-\delta_{\gamma(z)}.
\end{equation*}

\end{prop}\label{edeltaprop}
\begin{proof} Fix $z\in\Z_{\geq 0}$. For $s_{n(x)-1}\le x <s_{n(x)}$ we have that
\begin{equation*}
\delta_z(1_{(x)}) = \delta_z(1_{(n(x),x)}) = \left\{
\begin{aligned}
&1 &&\textrm{ for }s_{n(x)}|(z-x) \\
&0 &&\textrm{ else}.
\end{aligned}\right.
\end{equation*}
We want to determine for which $x$ we have $\delta_z(1_{(x)})$ is non-zero. First observe that to get nontrivial $\delta_z(1_{(x)})$, we must have $x\le z$ since otherwise $0<x-z<s_{n(z)}$ and so $s_{n(z)}$ cannot divide $x-z$.  In particular, for a fixed $z$, there are only finitely many $x$ with non-zero evalution. We determine them inductively in decreasing order. 

The first is $x=z$ and it is the only one such that $s_{n(x)}=s_{n(z)}$.  The next possible $s_{n(x)}$ is $s_{n(z)-1}$ and for such a choice we must have $s_{n(z)-1}$ dividing $x-z$ and hence 
$$x=z\textrm{ mod }s_{n(z)-1} = \gamma(z).$$ 
Repeating this reductive process results in the first formula.  The second formula is a direct consequence of the first.
\end{proof}

\subsection{Fredholm Modules}
An odd Fredholm module over a unital C$^*$-algebra $A$, see for example \cite{HR}, is a triple $(H, \rho , F)$ where $H$ is a Hilbert space, $\rho: A \to B(H)$ is a $*$-representation, and $F \in B(H)$ satisfies 
$$F^*- F,\ I - F^2 \in \mathcal{K}(H),\textrm{ and }[F, \rho(a)] \in \mathcal{K}(H)$$ 
for all $a \in A$, where $ \mathcal{K}(H)$ is the algebra of compact operators in $H$. Similarly, an even Fredholm module is the above information, together with a $\Z_2$-grading $\Gamma$ of $H$, such that 
$$\Gamma^* = \Gamma,\ \Gamma^2 = I,\ \Gamma F = - F \Gamma,\ \textrm{and  }\Gamma \rho(a) = \rho(a) \Gamma$$ 
for all $a \in A$.  It follows that $H=H_{ev} \oplus H_{odd}$, $\rho_{ev} \oplus \rho_{odd}$ . Without the loss of generality, we can assume that $F$ is self-adjoint, so it has the form:
\begin{equation}\label{FGmatrix}
 F = \begin{bmatrix} 0 & G \\G^* & 0 \end{bmatrix}
\end{equation}
where $G$ maps $H_{odd}$ to $H_{ev}$.
Classes of odd Fredholm modules over $A$ form the K-homology group $K^1(A) : = KK^1(A, \mathbb{C})$, while classes of even Fredholm modules over $A$ form the K-homology group $K^0(A) : = KK^0(A, \mathbb{C})$. 

For any C$^*$-algebra $A$ and $i=0,1$ there are pairings between $K_i(A)$ and $K^i(A)$, see \cite{Connes}. For our purpose we only need the pairing of the class in $K_0(A)$ of a projection $P$ in the algebra $A$ and a Fredholm module $(H_{ev} \oplus H_{odd}, \rho_{ev} \oplus \rho_{odd}, F)$. It is given by 
$$
\langle [P]_0 , (H_{ev} \oplus H_{odd}, \rho_{ev} \oplus \rho_{odd}, F) \rangle = \textrm{Index}(\rho_{ev}(P) G \rho_{odd}(P)), 
$$
where the index above is of the restricted operator
$$\rho_{ev}(P) G \rho_{odd}(P): \textrm{Ran}( \rho_{odd}(P)) \to \textrm{Ran}(\rho_{ev}(P))\,.$$ 

Given a homomorphism $\Phi\in \Hom(C(\Z_S, \Z),\Z)$ such that
\begin{equation}\label{PhiDecomp}
\Phi=\sum_{y=0}^\infty\phi_{(y)}\,e_{(y)}= \sum_{y=0}^\infty\phi_{(y)}(\delta_y-\delta_{\alpha(y)})
\end{equation}
we want to describe an even Fredholm module $(H_{ev} \oplus H_{odd}, \rho_{ev} \oplus \rho_{odd}, F)$ such that its pairing with $K_0(C(\Z_S))$ gives the homomorphism $\Phi$. Informally, we build the representation $\rho_{ev} \oplus \rho_{odd}$ from evaluations at $y$ and $\gamma(y)$ for those $y$ for which $\phi_{(y)}\ne 0$ and group the terms in equation \eqref{PhiDecomp} with positive coefficients to form $H_{odd}$ and the terms with negative coefficients to form $H_{ev}$.

More precisely, the Hilbert space for this module is a separable Hilbert space with distinguished basis labeled in the following way. $H_{odd}$ has basis $E^{{odd},+}_{(y,j)}$ for all $y$ such that $\phi_{(y)}>0$ and $j=1,\ldots,\phi_{(y)}$ and also $E^{{odd},-}_{(z,k)}$ for all $z$ such that $\phi_{(z)}<0$ and $k=1,\ldots,|\phi_{(z)}|$. Similarly, the distinguished basis in $H_{ev}$ is labeled $E^{{ev},+}_{(y,j)}$ for all $y$ such that $\phi_{(y)}>0$ and $j=1,\ldots,\phi_{(y)}$ and also $E^{{ev},-}_{(z,k)}$ for all $z$ such that $\phi_{(z)}<0$ and $k=1,\ldots,|\phi_{(z)}|$.

The representation $\rho_{ev}$ of $C(\Z_S)$ in $H_{ev}$ is defined on basis elements as follows:
\begin{equation*}
\rho_{ev}(f)E^{{ev},+}_{(y,j)} = f(\gamma(y)) E^{{ev},+}_{(y,j)} \textrm{ and } \rho_{ev}(f)E^{{ev},-}_{(z,k)}= f(z) E^{{ev},-}_{(z,k)}.
\end{equation*}
The definition of $\rho_{odd}$ is similar:
\begin{equation*}
\rho_{odd}(f)E^{{odd},+}_{(y,j)} = f(y) E^{{odd},+}_{(y,j)} \textrm{ and } \rho_{odd}(f)E^{{odd},-}_{(z,k)}= f(\gamma(z)) E^{{odd},-}_{(z,k)}.
\end{equation*}
Thus, with respect to the preferred basis, the representation of $C(\Z_S)$ in $H$ is a  direct sum of one-dimensional diagonal representations.

We also consider an operator $G: H_{odd}\to H_{ev}$ given on basis elements by:
\begin{equation*}
GE^{{odd},+}_{(y,j)} = E^{{ev},+}_{(y,j)} \textrm{ and } GE^{{odd},-}_{(z,k)}=E^{{ev},-}_{(z,k)}.
\end{equation*}
The definition implies that $G$ is an isometric isomorphism of Hilbert spaces:
\begin{equation*}
G^*G=I_{H_{odd}} \textrm{ and } GG^*=I_{H_{ev}}.
\end{equation*}

\begin{theo}\label{FredModTheo}
With the above notation $(H_{ev} \oplus H_{odd}, \rho_{ev} \oplus \rho_{odd}, F)$, where $F$ is given by the equation \eqref{FGmatrix}, is an even Fredholm module the class of which in $K^0(C(\Z_S))$ is the homomorphism $\Phi$.
\end{theo}
\begin{proof}
Since $G^*G$ and $GG^*$ are the identity on the respective spaces $H_{odd}$ and $H_{ev}$, it follows that $F^2=I$ and $F$ is self-adjoint.  Thus, we need to verify that $[F, \rho(f)] \in \mathcal{K}(H)$.  A direct calculation yields
\begin{equation*}
[F,\rho(f)] = \begin{pmatrix} 0 & G\rho_{odd}(f)-\rho_{ev}(f)G \\ G^*\rho_{ev}(f)-\rho_{odd}(f)G^* & 0\end{pmatrix}.
\end{equation*}
We only have to check that  $G\rho_{odd}(f)-\rho_{ev}(f)G$ is a compact operator from $H_{odd}$ to $H_{ev}$.  

Notice on the ``odd'' basis elements we have the following formulas:
\begin{equation*}
\begin{aligned}
&(G\rho_{odd}(f)-\rho_{ev}(f)G)E_{(y,j)}^{odd,+} = (f(y)-f(\gamma(y)))E_{(y,j)}^{ev,+}\\
&(G\rho_{odd}(f)-\rho_{ev}(f)G)E_{(z,k)}^{odd,-} = (f(\gamma(z))-f(z))E_{(z,k)}^{ev,-}
\end{aligned}.
\end{equation*}

We first study the above equation for $f=1_{(x)}$ and the ``odd/plus'' basis elements. By Proposition \ref{edeltaprop}, we have that
\begin{equation*}
\begin{aligned}
(1_{(x)}(y)-1_{(x)}(\gamma(y)))E_{(y,j)}^{ev,+} &= (\delta_y(1_{(x)})-\delta_{\gamma(y)}(1_{(x)}))E_{(y,j)}^{ev,+} = e_{(y)}(1_{(x)})E_{(y,j)}^{ev,+} \\
&=\left\{\begin{aligned}
&E_{(x,j)}^{ev,+} &&\textrm{ if }y=x\\
&0 &&\textrm{ else.}
\end{aligned}\right.
\end{aligned}
\end{equation*}
The same happens with the ``odd/minus'' basis elements except for a difference of a minus sign.  Thus $G\rho_{odd}(1_{(x)})-\rho_{ev}(1_{(x)})G$ is a finite rank operator and so, if $f$ is a linear combination of the $1_{(x)}$'s, then $G\rho_{odd}(f)-\rho_{ev}(f)G$ is a finite rank operator.  Since the linear combinations of the $1_{(x)}$'s are dense in $C(\Z_S)$ and the finite rank operators are norm dense in the compact operators, it follows that $G\rho_{odd}(f)-\rho_{ev}(f)G$ is a compact operator on $H_{odd}$. This establishes that $(H_{ev} \oplus H_{odd}, \rho_{ev} \oplus \rho_{odd}, F)$
 is an even Fredholm module over $C(\Z_S)$.

Finally we just need to verify that the index homomorphism for this Fredholm module coincides with $\Phi$. It is enough to verify that on the generators $1_{(x)}$'s of $K_0(C(\Z_S))$.  

Let the operator
\begin{equation*}
B_x:\textrm{Ran}(\rho_{odd}(1_{(x)}))\to\textrm{Ran}(\rho_{ev}(1_{(x)}))
\end{equation*}
be given by $B_x = \rho_{ev}(1_{(x)})G\rho_{odd}(1_{(x)})$.  We want to compute its index. We first look at its domain.

Using the definition of $\rho$ on the ``odd'' basis elements yields
\begin{equation*}
\rho_{odd}(1_{(x)})E_{(y,j)}^{odd,+} = \delta_y(1_{(x)})E_{(y,j)}^{odd,+}\quad\textrm{and}\quad \rho_{odd}(1_{(x)})E_{(z,k)}^{odd,-}=\delta_{\gamma(z)}(1_{(x)})E_{(z,k)}^{odd,-}.
\end{equation*}

Using Proposition \ref{edeltaprop} we see that
\begin{equation*}
\delta_y(1_{(x)})=(e_{(y)}+e_{(\gamma(y))}+\cdots+e_{(0)})(1_{(x)}) \textrm{ and } \delta_{\gamma(z)}(1_{(x)})= (e_{(\gamma(z))}+e_{(\gamma^2(z))}+\cdots+e_{(0)})(1_{(x)}).
\end{equation*}
Consequently, we have
\begin{equation*}
\textrm{Ran}(\rho_{odd}(1_{(x)})) = \textrm{span }(X_1\cup X_2)
\end{equation*}
where
\begin{equation*}
X_1 = \{E_{(y,j)}^{odd,+}:x=\gamma^l(y),\textrm{ for some }l\ge0\}\textrm{ and }X_2 =\{E_{(z,k)}^{odd,-}: x=\gamma^l(z),\textrm{ for some }l\ge1\}.
\end{equation*}

Calculating on the ``odd'' basis elements yields
\begin{equation*}
B_xE_{(y,j)}^{odd,+} = \delta_{\gamma(y)}(1_{(x)})\delta_y(1_{(x)})E_{(y,j)}^{ev,+}\quad\textrm{and}\quad B_xE_{(z,k)}^{odd,-}=\delta_z(1_{(x)})\delta_{\gamma(z)}(1_{(x)})E_{(z,k)}^{ev,-}.
\end{equation*}

If $x=y$, it follows that $e_{(\gamma(y))}(1_{(x)})=e_{(\gamma^2(y))}(1_{(x)})=\cdots=e_{(0)}(1_{(x)})=0$.  On the other hand, if $x=\gamma^l(y)$ for $l\ge1$ then $\delta_y(1_{(x)})\delta_{\gamma(y)}(1_{(x)}) \neq 0$.  Consequently we have that the kernel of $B_x$ restricted to the range of $\rho_{odd}(1_{(x)})$ is equal to the following set:
\begin{equation*}
\textrm{span }\{E_{(y,j)}^{odd,+}:j=1,\ldots \varphi_x\}.
\end{equation*}
Thus we have that
\begin{equation*}
\textrm{dim}(\textrm{ker}(B_x)) = \left\{
\begin{aligned}
&\varphi_x &&\textrm{ if }\varphi>0\\
&0 &&\textrm{ if }\varphi<0.
\end{aligned}\right.
\end{equation*}

A completely analogous calculation shows that
\begin{equation*}
\textrm{dim}(\textrm{ker}(B_x^*)) = \left\{
\begin{aligned}
&0 &&\textrm{ if }\varphi>0\\
&-\varphi_x &&\textrm{ if }\varphi<0.
\end{aligned}\right.
\end{equation*}
Putting these together shows that $\textrm{ind}(B_x)=\varphi_x$.
\end{proof}

\subsection{Spectral Triples}
The unbounded analogues of Fredholm modules are called spectral triples. They are important tools in noncommutative geometry. In the definition of a spectral triple, as compared to a Fredholm module, we usually suppress mentioning the Hilbert space $H$ as it is implicit in the definition of a representation $\rho$. The new element is a dense $*$-subalgebra $\mathcal{A}$ of a C$^*$-algebra $A$.  See reference \cite{Ren} for a discussion of smooth subalgebras associated to spectral triples. 

A spectral triple for a unital $C^*$-algebra $A$ is a triple $(\mathcal A, \rho, \mathcal D)$ where $\rho: A \rightarrow B(H)$ is a representation of $A$ on a Hilbert space $H$,  $\mathcal A$ is a dense $*$-subalgebra of $A$, and $\mathcal D$ is an unbounded self-adjoint operator in $\mathcal H$ satisfying: 

(1) for every $a \in \mathcal A$, $\rho(a)$ preserves the domain of $\mathcal D$,

(2) for every $a \in \mathcal A$, the commutator $[\mathcal D, \rho(a)]$ is bounded,

(3) the resolvent $(1+ \mathcal D^2)^{-1/2}$ is a compact operator.
\newline Moreover,  $(\mathcal A, \mathcal H, \mathcal D)$ is said to be an even spectral triple if there a $\Z_2$-grading of $H$, such that $H=H_{ev} \oplus H_{odd}$, $\rho_{ev} \oplus \rho_{odd}$ and 
\begin{equation*}
\mathcal D = \begin{bmatrix} 0 & D \\D^* & 0 \end{bmatrix}.
\end{equation*}

An even spectral triple for a unital $C^*$-algebra $A$ determines a homomorphism from $K_0(A)$ to $\Z$. For our purpose we only need the formula for this homomorphism for a projection $P$ in the algebra $A$. As for Fredholm modules it is given by the index 
\begin{equation*}
P\mapsto \textrm{Index}(\rho_{ev}(P) D \rho_{odd}(P)).
\end{equation*}

Suppose $\lambda=\lambda_{s,l}$ is a non-archimedean length functions on  $\widetilde{\Z_S}$ and let $C_{RD}(\Z_S)$ be the corresponding smooth subalgebra of the C$^*$-algebra $C(\Z_S)$.
Let $\Phi\in \Hom(C(\Z_S, \Z),\Z)$ be a homomorphism from equation \eqref{PhiDecomp}. Also, let  $H_{ev} \oplus H_{odd}$ and $\rho_{ev} \oplus \rho_{odd}$ be the Hilbert space and the representation from Theorem \ref{FredModTheo}. We use those objects to define a spectral triple over $C(\Z_S)$ corresponding to $\Phi$.

If $n(y)$ is such that $s_{n(y)-1}\leq y<s_{n(y)}$ we set for $y\geq 1$, $y\in\Z$:
\begin{equation*}
\Lambda(y):= l_{n(y)},
\end{equation*}
and set $\Lambda(0)=1$.

Consider an operator $D: \textrm{dom}\,D\subset H_{odd}\to H_{ev}$ given on basis elements by:
\begin{equation*}
DE^{{odd},+}_{(y,j)} = \Lambda(y) E^{{ev},+}_{(y,j)} \textrm{ and } DE^{{odd},-}_{(z,k)}=\Lambda(z) E^{{ev},-}_{(z,k)},
\end{equation*}
with the maximal domain $\textrm{dom }D$, see below. Because $l_{n}\to \infty$, the resolvent $(1+ \mathcal D^2)^{-1/2}$ is a compact operator.

\begin{theo}
With the above notation $(C_{RD}(\Z_S), \rho_{ev} \oplus \rho_{odd}, \mathcal{D})$ is an even spectral triple which gives the homomorphism $\Phi$ on $K_0(C(\Z_S))$.
\end{theo}
\begin{proof}
The above consideration already verifies the that $\mathcal{D}$ has compact resolvent.  Thus we just need to verify that $\rho(f)$ preserves the domain of $\mathcal{D}$ and that $[\mathcal{D}, \rho(f)] \in B(H)$ for $f\in C_{RD}(\Z_S)$.  

First for the domain preservation, by the definition of $\mathcal{D}$, it is enough to check that $\rho_{odd}(f)$ preserves the maximal domain of $D$ for $f\in C_{RD}(\Z_S)$.  By the definition of the maximal domain we have that
\begin{equation*}
\begin{aligned}
&\textrm{dom }D =\\
&\,\,= \left\{\Omega = \sum_{y,j}\omega_{(y,j)}^+E_{(y,j)}^{odd,+} + \sum_{z,k}\omega_{(z,k)}^-E_{(z,k)}^{odd,-} : \sum_{y,j}\Lambda^2(y)|\omega_{(y,j)}^+|^2 + \sum_{z,k}\Lambda^2(z)|\omega_{(z,k)}^-|^2<\infty\right\}.
\end{aligned}
\end{equation*}
The definition of $\rho_{odd}$ yields
\begin{equation*}
\rho_{odd}(f)\Omega = \sum_{y,j}\omega_{(y,j)}^+f(y)E_{(y,j)}^{odd,+} + \sum_{z,k}\omega_{(z,k)}^-f(\gamma(z))E_{(z,k)}^{odd,-}
\end{equation*}
for $f\in C_{RD}(\Z_S)$.  Since $f$ is bounded, the above calculation shows that $\rho_{odd}(f)\Omega\in\textrm{dom }D$.

Similarly to the proof of Theorem \ref{FredModTheo} to check that $[\mathcal{D},\rho(f)]$ is bounded for $f\in C_{RD}(\Z_S)$, we compute $D\rho_{odd}(f)-\rho_{ev}(f)D$  on the ``odd/plus'' basis elements.  Calculating directly we get
\begin{equation*}
(D\rho_{odd}(f) - \rho_{ev}(f)D)E_{(y,j)}^{odd,+} = \Lambda(y)(f(y)-f(\gamma(y)))E_{(y,j)}^{ev,+}.
\end{equation*}
Thus, to show boundedness, it is enough to show that
\begin{equation*}
\underset{y\ge0}{\textrm{sup }}\Lambda(y)|f(y) - f(\gamma(y))|<\infty
\end{equation*}
for $f\in C_{RD}(\Z_S)$.  

For such an $f$ we have that $f(y) = \sum_{z\in\widehat{\Z_S}}\hat{f}_z\chi_z(y)$ and so
\begin{equation*}
f(y) - f(\gamma(y)) = \sum_{z\in\widehat{\Z_S}}\hat{f}_z(\chi_z(y) -\chi_z(\gamma(y))).
\end{equation*}
Since $s_{n(y)-1}\le y<s_{n(y)}$ and so $y= \gamma(y) + rs_{n(y)-1}$ for some integer $r$, and thus  we have that
\begin{equation*}
\chi_z(y) - \chi_z(\gamma(y)) = \chi_z(\gamma(y))\left(z^{rs_{n(y)-1}}-1\right).
\end{equation*}

For $z\in\widetilde{\Z_S}$ let $m(z)$ be the smallest non-negative integer such that
\begin{equation*}
z^{s_{m(z)}}=1.
\end{equation*}
It follows that if  $s_{n(y)}>s_{m(z)}$ then $s_{m(z)}$ divides  $s_{n(y)}$ and in the formula above  we have $\chi_z(y) - \chi_z(\gamma(y)) =0$ for  $z\in\widetilde{\Z_S}$ such that $n(y)>m(z)$.
Consequently we get the following estimate
\begin{equation*}
\begin{aligned}
\Lambda(y)|f(y)-f(\gamma(y))| &=\left|l_{n(y)}\sum_{z\in\widetilde{\Z_S},\,m(z)\ge n(y)}\hat{f}_z\chi_z(\gamma(y))\left(z^{rs_{n(y)-1}}-1\right)\right| \\
&\le 2l_{n(y)}\sum_{z\in\widetilde{\Z_S},\, m(z)\ge n(y)}|\hat{f}_z|\le 2\sum_{z\in\widetilde{\Z_S},\, m(z)\ge n(y)}l_m|\hat{f}_z|\\
&\le 2\sum_{z\in\widehat{\Z_S}}\Lambda(z)|\hat{f}_z| = 2\|f\|_1.
\end{aligned}
\end{equation*}

Finally, the index calculation for the spectral triple is completely analogous to the index calculation for the Fredholm module in the proof of Theorem \ref{FredModTheo} as the extra non-zero $\Lambda(y)$ factor does not change the kernel.  
\end{proof}

\end{document}